\documentclass[10pt]{amsart}
\usepackage{amssymb}
\usepackage{tikz-cd}
\usepackage{amsrefs}
\usepackage{xfrac}
\usepackage{thmtools}
\usepackage{thm-restate}

\title[A Classification of Order Convergence]{A Classification of Order Convergence via a Transfinite Fatou Hierarchy}

\author{A.~Avil\'es}
\address{Universidad de Murcia, Departamento de Matem\'aticas, Campus de Espinardo 30100 Murcia, Spain.}
\email{avileslo@um.es}

\author {C.~Rosendal}
\address{Department of Mathematics\\University of Maryland\\4176 Campus Drive - William E. Kirwan Hall\\College Park, MD 20742-4015\\USA}
\email{rosendal@umd.edu}
\urladdr{sites.google.com/view/christian-rosendal/}

\author{M.~A.\ Taylor}
\address{Department of Mathematics\\
ETH Z\"urich, Ramistrasse 101, 8092 Z\"urich, Switzerland.
} \email{mitchell.taylor@math.ethz.ch}

\author{P.~Tradacete}
\address{Instituto de Ciencias Matem\'aticas (CSIC-UAM-UC3M-UCM)\\
Consejo Superior de Investigaciones Cient\'ificas\\
C/ Nicol\'as Cabrera, 13--15, Campus de Cantoblanco UAM\\
28049 Madrid, Spain.}
\email{pedro.tradacete@icmat.es}

\newcommand{\norm}[1]{\lVert#1\rVert}
\newcommand{\Norm}[1]{\big\lVert#1\big\rVert}
\newcommand{\NORM}[1]{\Big\lVert#1\Big\rVert}
\newcommand{\BNORM}[1]{\Bigg\lVert#1\Bigg\rVert}

\newcommand{\forkindep}[1][]{\mathop{\mathop{\vcenter{\hbox{\oalign{\noalign{\kern-.3ex}
\hfil$\vert$\hfil\cr\noalign{\kern-.7ex}$\smile$\cr\noalign{\kern-.3ex}}}}}\displaylimits_{#1}}}

\newcommand{\Mgd}[2]{\big\{\,  {#1}\;\big|\; {#2} \,\big\} }
\newcommand{\MGd}[2]{\Big\{ \, {#1}\;\Big|\; {#2} \,\Big\} }

\newcommand{\maths}[1]{\[\begin{split}{#1}\end{split}\]}

\newcommand{\conv}[2]{\mathop{\underset{#2}{\overset{#1}\longrightarrow}}}

\newcommand{\maps}[1]{\mathop{\overset{#1}\longrightarrow}}

\newcommand{\finitegames}[8]{
$$
{
\begin{array}{ccccccccccccccc}
{\bf I} &&{#1}&       & {#3}&      &&\ldots  &&{#5}&        &{#7}&& \\
{\bf II}&&       &{#2}&        &{#4} & &\ldots   &&  &{#6}&&{#8}&
\end{array}
}
$$
}

\usetikzlibrary{decorations.markings}
\tikzset{negated/.style={
        decoration={markings,
            mark= at position 0.5 with {
                \node[transform shape] (tempnode) {$\backslash$};
            }
        },
        postaction={decorate}
    }
}

\newcommand {\N}{\mathbb N}

\newcommand {\Q}{\mathbb Q}
\newcommand {\R}{\mathbb R}

\newcommand{\om}{\omega}
\newcommand{\eps}{\epsilon}

\newcommand{\tom} {\emptyset}

\newcommand{\saa}{\Rightarrow}
\newcommand{\equi}{\Leftrightarrow}

\newcommand {\Del}{ \; \Big| \;}
\newcommand {\del}{ \; \big| \;}

\newcommand {\e} {\exists}
\renewcommand {\a} {\forall}

\theoremstyle{plain}
\newtheorem{thm}{Theorem}[section]
\newtheorem*{theorem*}{Theorem}

\newtheorem{lemme}[thm]{Lemma}
\newtheorem{prop} [thm] {Proposition}
\newtheorem{defi} [thm] {Definition}

\theoremstyle{definition}

\newtheorem{rem}[thm]{Remark}
\newtheorem{exa}[thm]{Example}

\definecolor{groen}{rgb}{0,0.5,.7}
\definecolor{gul}{rgb}{0.94,0.8,0}
\definecolor{blaa}{rgb}{0.16,0,0.6}
\definecolor{roed}{rgb}{1,0,0}

\makeindex

\thanks{C.~Rosendal~was partially supported by the U.S.~National Science Foundation under Award Numbers DMS-2246986 and DMS-2204849. A.~Avil\'{e}s was supported by MICIU/AEI /10.13039/501100011033/ and ERDF-A way of making Europe (project PID2021-122126NB-C32). A.~Avil\'{e}s and P.~Tradacete were supported by Fundaci\'{o}n S\'{e}neca - ACyT Regi\'{o}n de Murcia. P.~Tradacete was partially supported by grants PID2020-116398GB-I00, PID2024-162214NB-I00 and CEX2023-001347-S funded by MCIN/AEI/10.13039/501100011033, as well as by a 2022 Leonardo Grant for Researchers and Cultural Creators, BBVA Foundation.}

\begin{document}

\begin{abstract}
    We investigate the descriptive complexity of order convergence in separable Banach lattices. While uniform convergence is Borel and $\sigma$-order convergence is known to be ${\bf \Delta}^1_2$, it is unclear in general when $\sigma$-order convergence is analytic.

We introduce a transfinite hierarchy of weakenings of the classical Fatou property, indexed by countable ordinals, and show that it provides a complete structural classification of this definability problem. For a separable Banach lattice 
$X$, we prove that the following are equivalent: (i) the set of decreasing positive sequences with infimum zero is Borel; (ii) $\sigma$-order convergence is analytic; and (iii) 
$X$ satisfies the $\alpha$-Fatou property for some countable ordinal $\alpha$.

We further establish that the hierarchy is proper: for every countable ordinal $\alpha$ there exists a separable Banach lattice with a countable $\pi$-basis that fails to be $\alpha$-Fatou, but is $\beta$-Fatou for some $\beta>\alpha$. Thus the Borel definability of order convergence is governed by a canonical ordinal invariant intrinsic to the lattice, and the descriptive complexity can be arbitrarily high below $\om_1$.

These results identify projective complexity as a genuine structural invariant in Banach lattice theory.
\end{abstract}

\maketitle

\section{Introduction}
The present paper develops a descriptive set theoretic classification of order convergence in separable Banach lattices. Our goal is to measure, in a precise and intrinsic way, the definability complexity of the most natural order-theoretic convergence relations and to identify the structural lattice properties governing this complexity.

A {\em Banach lattice} is a Banach space $X$ equipped with a lattice ordering $\leqslant$ that is compatible with both the linear and norm structure on $X$. Specifically,  for all $x,y,z\in X$ and $\lambda>0$, 
$$
x\leqslant y\; \Rightarrow\; \lambda x+z\leqslant \lambda y+z
\qquad\&\qquad 
|x|\leqslant |y|\;\Rightarrow \;\norm{x}\leqslant \norm{y},
$$
where $|x|=x\vee -x$. These axioms imply that the order relation  $\leqslant$ is norm closed in $X^2$.  

However, many other natural order-theoretic constructions exhibit dramatically higher descriptive complexity. For example, the sets
$$
X^{\uparrow\infty} =\Mgd{(x_n)\in X^\N}{0\leqslant x_1\leqslant x_2\leqslant \ldots \text{ and $(x_n)$ is order-unbounded in }X}$$
and 
$$ 
X_{\downarrow 0} =\Mgd{(x_n)\in X^\N}{x_1\geqslant x_2\geqslant \ldots \geqslant 0 =\inf_nx_n}
$$
can easily be seen to be coanalytic, i.e., ${\bf \Pi}^1_1$, but do not appear to be Borel in $X^\N$. 

This observation raises a fundamental structural question:
\begin{quote}
    How complicated, in the projective hierarchy, are the natural order convergence relations in a separable Banach lattice?
\end{quote}
Our interest in this question stems in part from the study of coordinate systems and bases in Banach lattices \cites{ARTT, taylor, gumenchuk}, where definability properties of order convergence play a decisive role. Every Banach lattice supports several distinct notions of sequential convergence.
\begin{itemize}
\item A sequence $(x_n)_{n=1}^\infty$ {\em converges uniformly} to $x$, denoted $x_n\maps{\sf u}x$, if there is some $z\in X$  so that
$$
\a m \;\a^\infty n\; |x_n-x|\leqslant \tfrac zm,
$$
\item a sequence $(x_n)_{n=1}^\infty$ {\em $\sigma$-order converges} to $x$, denoted $x_n\maps{\sf \sigma \sf o}x$, if there is some sequence $z_m\downarrow 0$ in $X$ so that
$$
\a m \;\a^\infty n\; |x_n-x|\leqslant z_m,
$$
\item a sequence $(x_n)_{n=1}^\infty$ {\em order converges} to $x$, denoted $x_n\maps{\sf o}x$, if there is some net $z_{\mu}\downarrow 0$ in $X$ so that
$$
\a  \mu\;\a^\infty n\; |x_n-x|\leqslant z_\mu.
$$
\end{itemize}
Here the notations $z_m\downarrow 0$ and $z_\mu\downarrow 0$ mean that $(z_m)$ is a decreasing sequence, respectively decreasing net, with infimum $0$, while the notation $\a^\infty n$ means for all but finitely many $n$. Figure \ref{diagram1} lists the relationships holding between these sequential convergence notions and it should be noted that no other implications hold in general Banach lattices.

\begin{figure}[!htb]\label{diagram1}
\begin{tikzcd}
&{\sf uniform} \arrow[-{Implies},double]{dr}\arrow[-{Implies},double]{dl}&\\ 
\sigma-{\text{}\sf order}\arrow[-{Implies},double]{d}{} &&     {\sf norm}\arrow[-{Implies},double]{d}{}  \\
  {\text{}\sf order} &&{\sf weak}\\
\end{tikzcd}
\caption{Implications between convergence types.}
\end{figure}
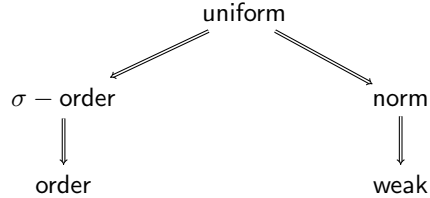
\begin{rem}
    It is a classical fact (see \cite[Remark 1.3]{Tay}) that, if $X$ is a Banach lattice of measurable functions (i.e.~an ideal in the space of measurable functions $L_0(\Omega,\Sigma,\mu)$ for some semi-finite measure space $(\Omega,\Sigma,\mu)$), we have $f_n\conv{\sf o}{n}f$ if and only if $f_n\conv{a.e.}{n}f$ and, moreover,  there exists a $g\in X_+$ satisfying  $|f_n|\leqslant g$ for all $n$ (i.e., the sequence $(f_n)$ is order bounded). Therefore, order convergence can be viewed as a generalization of dominated almost everywhere convergence to vector lattices. On the other hand, if $e\in X^+$, then we may equip the order ideal $I_e=\Mgd{x\in X }{ \exists \lambda\in \mathbb{R}\  |x|\leq \lambda e}$ with the norm $\|\cdot\|_e$ where for each $x\in I_e,$ $\|x\|_e$ is the smallest $\lambda=\lambda(x)$ which makes the definition of $I_e$ valid. It is well-known that $(I_e,\|\cdot\|_e)$ is lattice isometric to $C(K_e)$ for some compact Hausdorff space $K_e$. Moreover, in $C(K_e)$, $f_n\conv{\sf u}{n}f$ if and only if $f_n\conv{\|\cdot\|_\infty}{n}f$, so uniform convergence bonds the local norm convergences in $C(K_e)$ to a convergence in all of $X$. These convergences have numerous applications, and have been studied extensively in recent years \cite{MR4366912,MR3731703,MR3818109}.
\end{rem}
At first glance, the above convergences appear to have high descriptive complexity. For example, $\sigma$-order convergence is naturally expressible as a ${\bf \Sigma}^1_2$, meaning that the corresponding subset of $X^\N\times X$ is the continuous image of a ${\bf \Pi}^1_1$ subset of a Polish space. On the other hand, it is unclear whether the relation $x_n\maps{\sf o}x$ is even projective, i.e., ${\bf \Sigma}^1_k$ for some $k$. Yet, as shown in \cite{ARTT}, in separable Banach lattices order and $\sigma$-order convergence coincide and the resulting relation is ${\bf \Delta}^1_2$, that is, both ${\bf \Sigma}^1_2$ and ${\bf \Pi}^1_2$. In contrast, uniform convergence is  Borel rather than the prima facie value ${\bf \Sigma}^1_1$.

The fact that $x_n\maps{\sigma\sf o}x$ is ${\bf \Delta}^1_2$ but not known to be ${\bf \Sigma}^1_1$ is a cause of some difficulties in \cite{ARTT}. Namely, to show that every { $\sigma$-order basis} $(e_n)$ for a Banach lattice $X=[e_n]$ is also a Schauder basis, additional set theoretical assumptions are needed. However, this need vanishes once the relation $x_n\maps{\sigma\sf o}x$ can be shown to be ${\bf \Sigma}^1_1$.

We will here define a transfinite hierarchy of successive weakenings of the {\em Fatou property}  for Banach lattices, that is, the property that whenever we have elements $0\leqslant x_1\leqslant x_2\leqslant \ldots \leqslant x$ with $x=\sup_nx_n$, then $\norm x=\sup_n\norm {x_n}$. In turn, our first main result shows that this hierarchy exhausts the class of separable Banach lattices in which the set $X_{\downarrow0}$ is Borel or equivalently the relation $x_n\maps{\sigma\sf o}x$ is analytic.

\begin{thm}\label{thm:Fatou}
The following  are equivalent for a separable Banach lattice $X$.
\begin{enumerate}
\item The set 
$$
X_{\downarrow0}=\Big\{(x_n)_{n=1}^\infty\in X^\N\Del x_1\geqslant x_2\geqslant \ldots\geqslant 0=\inf_nx_n\Big\}
$$
is Borel,
\item the set 
$$
\Big\{ \big((x_n)_{n=1}^\infty,x\big)\in X^\N\times X\Del x_n\conv{\sigma\sf o}{n}x\Big\}
$$ 
is  analytic,
\item $X$ is $\alpha$-Fatou for some $\alpha<\om_1$.
\end{enumerate}
\end{thm}
The ordinal invariant introduced here plays a role analogous to classical indices in Banach space theory, such as the Szlenk index and Bourgain’s $\ell_1$--index, in that it measures structural complexity via well-founded ranks.

As we shall see later on, there are several natural conditions on $X$ that force $X_{\downarrow0}$ to be Borel. For example, this is the case when $X$ is either $\sigma$-order continuous, $\sigma$-order complete or has a countable $\pi$-basis $\{b_n\}_{n=1}^\infty$, that is, elements $b_n>0$ such that every $x>0$ is minorized by some $b_n$.  Similarly, $X^{\uparrow\infty}$ is Borel when $X$ is $\sigma$-order continuous, $\sigma$-order complete or has a {\em countable dominating set} $\{d_n\}_{n=1}^\infty$, that is, such that every $x\in X$ is dominated by some $d_n$.

The second main result is that the hierarchy of $\alpha$-Fatou properties is proper, in the sense that it is not exhausted at any level below $\om_1$.

\begin{thm}\label{thm:Fatou2}
For every countable ordinal $\alpha$, there is a separable Banach lattice with a countable $\pi$-basis that fails to be $\alpha$-Fatou.
\end{thm}
Observe that, having countable $\pi$-bases, the spaces of Theorem \ref{thm:Fatou2} must be $\beta$-Fatou for some $\beta>\alpha$. Although we do not have an example of a separable Banach lattice $X$ in which $X_{\downarrow0}$ fails to be Borel, the techniques behind Theorem \ref{thm:Fatou2} imply that there can be no uniform Borel manner of defining the property $x_n\downarrow 0$ across all separable Banach lattices $X$.


\section{Definability of order convergences}
As mentioned in the introduction, the prima facie complexities of uniform and order convergence are higher than their actual complexities. So we begin by recalling the better estimates obtained in \cite{ARTT}.

\begin{prop}[Lemmas 2.1 \& 2.3 \cite{ARTT}]\label{prop:complexity}
For a sequence $(x_n)$ and a vector $x$ in a Banach lattice $X$, we have that
\begin{enumerate}
   \item $x_n\conv{\sf u}{n}x$ if and only if
$$
\a \eps>0\; \e k\; \a m\geqslant k\quad \BNORM{\bigvee_{n=k}^m|x_n-x|}<\eps,
$$ 
\item $x_n\conv{\sf o}{n}x$ if and only if
\begin{equation*}
    \a y>0\;\e z\; \big(y\not\leqslant z \;\;\&\;\; \a^\infty n\; |x_n-x|\leqslant z\big).
\end{equation*} 
\end{enumerate}
\end{prop}
Recall on the other hand that the relation $x_n\conv{\sigma\sf o}{n}x$ is easily seen to be ${\bf \Sigma}^1_2$, since 
\maths{
x_n\conv{\sigma\sf o}{n}x
&\quad\equi\quad \e (z_m)\in X_{\downarrow 0}\; \a m\; \a^\infty n \; |x_n-x|\leqslant z_m
}
and $X_{\downarrow 0}$ is ${\bf \Pi}^1_1$.

\begin{prop}[Proposition 2.4 \cite{ARTT}]\label{prop conv agree}
Let $X$ be a separable Banach lattice. Then a sequence $(x_n)$ in $X$ order converges to $x\in X$ if and only if it $\sigma$-order converges to $x$.  In particular, the set 
\begin{equation}\label{compl of sigma-o}
    \Big\{ \big((x_n)_{n=1}^\infty,x\big)\in X^\N\times X\Del x_n\conv{\sigma\sf o}{n}x\Big\}
\end{equation}
is ${\bf \Delta}^1_2$, whereas
$$
\Big\{ \big((x_n)_{n=1}^\infty,x\big)\in X^\N\times X\Del x_n\conv{\sf u}{n}x\Big\}
$$ 
is Borel.
\end{prop}

The next step is to identify large and familiar classes of Banach lattices in which the sets $X_{\downarrow 0}$ and $X^{\uparrow \infty}$ are Borel rather than merely ${\bf \Pi}^1_1$. For this, let us recall that a Banach lattice $X$ is {\em $\sigma$-complete} if every countable sequence $(x_n)$ that is bounded above has a least upper bound denoted  $\sup_nx_n$. Also, the lattice $X$ is said to be {\em $\sigma$-order continuous} provided that $\lim_n\norm{x_n}=0$ for all  $(x_n)\in X_{\downarrow 0}$. Finally, we say that $X$ has the {\em Fatou property} if, whenever
$$
0\leqslant x_1\leqslant x_2\leqslant \ldots\leqslant x=\sup_n x_n,
$$
then $\norm{x}=\sup_n\norm{x_n}$. It is immediate to see that $\sigma$-order continuity implies the Fatou property, since the former forces that actually $\lim_n\norm{x-x_n}=0$ when $0\leqslant x_1\leqslant x_2\leqslant \ldots\leqslant x=\sup_n x_n$.

Now, by \cite[Propositions 1.a.7 and 1.a.8]{LT2} the following three properties are equivalent in the case of separable Banach lattices $X$.
\begin{enumerate}
\item $X$ is $\sigma$-complete,
    \item $X$ is $\sigma$-order continuous,
    \item every order bounded increasing sequence is norm convergent.
\end{enumerate}

\begin{lemme}
    If $X$ is a $\sigma$-complete separable Banach lattice, the sets $X_{\downarrow 0}$ and $X^{\uparrow \infty}$ are both Borel and the relation $x_n\conv{\sigma\sf o}{n}x$ is ${\bf \Sigma}^1_1$.
\end{lemme}

\begin{proof}
Observe that, by $\sigma$-order continuity, 
$$
X_{\downarrow 0}=\Mgd{(x_n)\in X^\N}{x_1\geqslant x_2\geqslant \ldots \geqslant 0\;\;\&\;\; \lim_n\norm{x_n}=0},
$$
whereas (3) above implies that
$$
X^{\uparrow \infty}=\Mgd{(x_n)\in X^\N}{0\leqslant x_1\leqslant x_2\leqslant \ldots \;\;\&\;\; \lim_{n,m\to \infty}\norm{x_n-x_m}\neq 0}.
$$
That $x_n\conv{\sigma\sf o}{n}x$ is ${\bf \Sigma}^1_1$ follows immediately from $X_{\downarrow 0}$ being Borel.
\end{proof}
As we shall see in Example \ref{example:fatou}, for $X_{\downarrow 0}$ to be Borel and $x_n\conv{\sigma\sf o}{n}x$ to be  ${\bf \Sigma}^1_1$, it is enough to assume that $X$ has the Fatou property. The same is true  if $X$ admits a countable {\em $\pi$-basis}  $\{b_n\}\subseteq X^+\setminus \{0\}$, meaning that, for every $x>0$, there in an $n$ so that $x>b_n$. Indeed, in this case,
$$
X_{\downarrow 0}=\Mgd{(x_n)\in X^\N}{x_1\geqslant x_2\geqslant \ldots \geqslant 0\; \;\&\;\; \a m\;\e n \; x_n\not> b_m  }
$$
is Borel.
Similarly, if $X$ admits a countable dominating family, then $X^{\uparrow \infty}$ is Borel.


\section{An ordinal index}
We now turn to other weaker conditions than the Fatou property or the existence of a countable $\pi$-basis. So,  in the following, let $X$ denote a fixed separable Banach lattice. Let us begin by recalling the following simple calculation.
\begin{lemme}
For all $x,y\in X$, we have 
$$
(y-x)^+=y-(x\wedge y).
$$
\end{lemme}

\begin{proof}
Recall that $+$ distributes over the lattice operation $\vee$, so 
\maths{
y-(x\wedge y)
&=y+\big((-x)\vee (-y)\big) 
=(y-x)\vee 0
=(y-x)^+
}
for all $x,y\in X$.
\end{proof}

If $\Sigma$ is any set, we let $\Sigma^{<\N}$ denote the set of all finite strings of elements of $\Sigma$. Recall that a subset $T\subseteq \Sigma^{<\N}$ is a {\em tree} provided that $T$ contains the empty string $\tom$ and is closed under taking initial segments. Recall also that a tree $T$ is said to be {\em ill-founded} provided that it has an {\em infinite branch}, i.e., if there is an infinite sequence $(y_n)_{n=1}^\infty$ in $\Sigma$ so that
$$
(y_1,\ldots, y_n)\in T
$$
for all $n$. Otherwise, $T$ is {\em well-founded}. For a well-founded tree $T$, we define an ordinal valued rank function $\rho_T\colon T\to {\sf Ord}$ by letting 
$$
\rho_T(s)=0  \;\equi\; s \text{ has no proper extensions in }T
$$
and otherwise
$$
\rho_T(s)=\sup\big\{ \rho_T(t)+1\del t\in T\;\&\; s\subsetneq t\}.
$$
We also define the {\em rank} of $T$ itself by
$$
\rho(T)=\sup\big\{ \rho_T(t)+1\del t\in T\}=\rho_T(\emptyset)+1.
$$

A binary relation $\prec$ on a set $\Omega$ is said to be {\em well-founded} if there is no infinite sequence of elements $p_n\in \Omega$ so that
$$
\ldots\prec p_3\prec p_2\prec p_1.
$$
In this case, we may similarly define a rank $\rho_\prec\colon \Omega\to {\sf Ord}$ by
$$
\rho_\prec(p)=0 \;\equi\; p \text{ is minimal, i.e., }q\not \prec p \text{ for all  }q\in \Omega
$$
and 
$$
\rho_\prec(p)=\sup\big\{ \rho_\prec(q)+1\del q\prec p\}.
$$
As for trees, we set $\rho(\prec)=\sup\big\{ \rho_\prec(p)+1\del p\in \Omega\}$. For example, if $T$ is a well-founded tree, we may let $\prec$ be the relation $\supsetneq$ on $T$, that is, for $s,t\in T$, we have $s\prec t$ if $t$ is a proper initial segment of $s$. Then $\rho_T=\rho_\prec$. 

We let $Tr_\Sigma$ denote the set of all trees $T\subseteq \Sigma^{<\N}$ and $WF_\Sigma$ denote the subset of all well-founded trees. If $\Sigma$ is a countable set, then $Tr_\Sigma$ is a closed subset of the Polish space $\{0,1\}^{\Sigma^{<\N}}$, whereas 
 $WF_\Sigma$ is a coanalytic subset of $Tr_\Sigma$. Moreover, by \cite[Exercise 34.6]{Kechris}, $\rho$ is a coanalytic rank on  $WF_\Sigma$. In particular, this means that, for all $\lambda<\om_1$, the set
$$
WF_\Sigma^\lambda=\{T\in Tr_\Sigma\del \rho(T)\leqslant \lambda\}
$$
is Borel.

Set
$$
X_\downarrow=\big\{(z_n)_{n=1}^\infty\in X^\N\del z_1\geqslant z_2\geqslant \ldots\geqslant 0\big\}
$$
and define, for every sequence $(z_n)\in X_\downarrow$, a tree $ \Psi \big((z_n)\big)\subseteq \big(X^+\!\setminus \{0\}\big)^{<\N}$ by
\maths{
(y_1,\ldots, y_k)\in  \Psi &\big((z_n)\big) \equi\\
&\a i\; \a n\quad \norm{(y_i-z_n)^+}
\leqslant\tfrac{\norm{y_i}}{3^{i}}
\; \;\&\;\; 
\norm{y_{i+1}-y_i}\leqslant \tfrac{\norm{y_i}}{3^{i}}.
}

\begin{lemme}\label{WF}
Let $P\subseteq X^+\setminus \{0\}$ be a countable norm-dense subset. Then, for all $(z_n)_{n=1}^\infty\in X_\downarrow$, the following conditions are equivalent.
\begin{enumerate}
\item $\inf_n z_n=0$,
\item $\Psi \big((z_n)\big)$ is well-founded,
\item $\Psi \big((z_n)\big)\cap P^{<\N}$ is well-founded.
\end{enumerate}
\end{lemme}

\begin{proof}Note first that the implication (2)$\saa$(3) is trivial. 

(3)$\saa$(1): Assume that $0$ is not the infimum of the sequence $(z_n)$, that is, that there is some  $y\in X^+$ satisfying
$$
 z_1\geqslant z_2\geqslant \ldots\geqslant y> 0.
 $$
Pick then elements $y_{i}\in  P$ satisfying $\norm{y_i-y}<\tfrac{\norm{y}}{7^{i}}$, which implies that
$$
\norm{y_{i+1}-y_i}\leqslant \norm{y_{i+1}-y}+\norm{y-y_i}<\frac{\norm{y}}{7^{i+1}}+\frac{\norm{y}}{7^{i}}<\frac{\norm{y_i}}{3^{i}}
$$
for all $i$. Furthermore,
$$
\norm{(y_i-z_n)^+}\leqslant \norm{(y-z_n)^+}+\norm{y-y_i}=\norm{y-y_i}<\frac{\norm{y_i}}{3^{i}},
$$
showing that $(y_1,y_2,\ldots)$ is an infinite branch of $\Psi\big((z_n)\big)\cap P^{<\N}$.

(1)$\saa$(2): 
Suppose that $\Psi \big((z_n)\big)$ is ill-founded, i.e., that $\Psi \big((z_n)\big)$ has an infinite branch $(y_1,y_2,\ldots)$. Then the conditions $ \norm{y_{i+1}-y_i}\leqslant\frac{\norm{y_i}}{3^{i}}$ ensure that $(y_i)$ is a Cauchy sequence converging to some $y>0$. Furthermore, as $\norm{(y_i-z_n)^+}\leqslant \frac{\norm{y_i}}{3^{i}}$ for all $n$ and $i$, we find that also
$$
\Norm{y-(z_n\wedge y)}=\Norm{(y-z_n)^+}=0,
$$
i.e., that $y\leqslant z_n$ for all $n$. So, $y$ is a strictly positive lower bound for $(z_n)$.
\end{proof}

\begin{lemme}\label{lemma:BONS} 
The  following conditions are equivalent.
\begin{enumerate}
\item The set 
$$
X_{\downarrow0}=\big\{(z_n)_{n=1}^\infty\in X_{\downarrow}\del \inf_nz_n=0\big\}
$$
is Borel,
\item 
$$
\sup\Big\{\rho\Big(\Psi \big((z_n)\big)\Big)\Del {(z_n)_{n=1}^\infty\in X_{\downarrow0}}\Big\}<\om_1.
$$
\end{enumerate}
\end{lemme}

\begin{proof}
(1)$\saa$(2): Note that, if $X_{\downarrow0}$ is Borel, so is the set
$$
\Omega
=\MGd{
\big((z_n)_{n=1}^\infty,(y_1,\ldots,y_m)\big)\in X_{\downarrow0}\times \big(X^+\!\setminus \{0\}\big)^{<\N} 
}
{
(y_1,\ldots,y_m)\in \Psi\big((z_n)\big)
}.
$$
We may then define a Borel quasiordering $\prec$ on $\Omega$ by setting
\maths{
\big((z_n)_{n=1}^\infty,&(y_1,\ldots,y_m)\big)\prec \big((u_n)_{n=1}^\infty,(v_1,\ldots,v_k)\big)\\
&\equi\; (z_n)_{n=1}^\infty=(u_n)_{n=1}^\infty 
\quad\&\quad
k<m
\quad\&\quad 
(v_1,\ldots,v_k)=(y_1,\ldots,y_k)
}
and observe that $\prec$ is well-founded by Lemma \ref{WF}. It then follows from the boundedness theorem for analytic well-founded relations \cite[Theorem 31.1]{Kechris} that
\maths{
\sup\Big\{\rho\Big(\Psi \big((z_n)\big)\Big)\Del {(z_n)_{n=1}^\infty\in X_{\downarrow0}}\Big\}
&=\sup\Big\{\rho_{\Psi \big((z_n)\big)}(\tom)+1\Del {(z_n)_{n=1}^\infty\in X_{\downarrow0}}\Big\} \\
&=\sup\Big\{\rho_\prec\big((z_n),\emptyset\big)+1\Del {(z_n)_{n=1}^\infty\in X_{\downarrow0}}\Big\} \\
&=\rho(\prec)\\
&<\om_1.
}

(2)$\saa$(1):  Observe that, if $P\subseteq X^+\setminus \{0\}$ is  a fixed countable norm-dense subset and 
$$
\sup\Big\{\rho\Big(\Psi \big((z_n)\big)\Big)\Del {(z_n)_{n=1}^\infty\in X_{\downarrow0}}\Big\}<\om_1,
$$
then also 
$$
\lambda=\sup\Big\{\rho\Big(\Psi \big((z_n)\big)\cap P^{<\N}\Big)\Del {(z_n)_{n=1}^\infty\in X_{\downarrow0}}\Big\}<\om_1.
$$
Note now that the map $X_\downarrow\maps \Theta  Tr_P$ defined by 
$$
\Theta \big((z_n)\big)= \Psi \big((z_n)\big)\cap P^{<\N}
$$
is Borel measurable and satisfies 
$$
(z_n)\in X_{\downarrow0}\quad\equi\quad \Theta \big((z_n)\big)\in WF^\lambda_P.
$$
Because $WF^\lambda_P$ is Borel, this shows that also $X_{\downarrow0}$ is Borel.
\end{proof}

For every ordinal $\alpha$ and $(z_n)_{n=1}^\infty\in X_\downarrow$, we define a game $G_\alpha\big[(z_n)\big]$ between two players I and II as follows. Players I and II alternate in playing ordinals  $\beta_i$  and  vectors $y_i\in X^+\!\setminus \{0\}$, 
\finitegames 
{\beta_1}{y_1}{\beta_2}{y_2}{\beta_{k-1}}{y_{k-1}}{\beta_k}{y_k}
and where the ordinals played are subject to the condition
$$
\alpha>\beta_1>\beta_2>\ldots>\beta_{k-1}>\beta_k\geqslant 0.
$$
The game ends when I  plays  $\beta_k=0$ and II plays its response  $y_k$. This will eventually happen as the ordinals are well-ordered. Player II is then said to {\em win} a run of the game provided that 
$$
\norm{y_{i+1}-y_i} \leqslant \frac{\norm{y_i}}{3^i}  
$$
for all $1\leqslant i<k$ and 
$$
\norm{(y_i-z_n)^+}\leqslant \frac{\norm{y_i}}{3^i}
$$
for all $n\geqslant 1$ and $1\leqslant i\leqslant k$. Otherwise player I wins.

\begin{exa}[Banach lattices with the Fatou property]\label{example:fatou}
Suppose $X$ is a Banach lattice with  the Fatou property, that is, whenever we have elements $0\leqslant x_1\leqslant x_2\leqslant \ldots \leqslant x$ with $x=\sup_nx_n$, then $\norm x=\sup_n\norm {x_n}$. Assume also that $(z_n)\in X_{\downarrow0}$. Then it is easy to see that I has a winning strategy in the game $G_1\big[(z_n)\big]$. Indeed, I simply plays $\beta_1=0$, to which II responds with some vector $y_1$. If II wins this run of the game, we must have 
$$
\norm{y_1-(z_n\wedge y_1)}=\norm{(y_1-z_n)^+}\leqslant \frac{\norm{y_1}}{3}
$$
for all $n$, while at the same time
$$
0\leqslant y_1-(z_1\wedge y_1)\leqslant y_1-(z_2\wedge y_1)\leqslant y_1-(z_3\wedge y_1)\leqslant \ldots \leqslant y_1= \sup_n\big(y_1-(z_n\wedge y_1)\big).
$$
Taken together, this contradicts the Fatou property. This means that a winning strategy for I in $G_1\big[(z_n)\big]$ is simply to play $\beta_1=0$.
\end{exa}

\begin{lemme}
For every $(z_n)_{n=1}^\infty\in X_\downarrow$ and every ordinal $\alpha$, 
\maths{
\rho\Big(\Psi \big((z_n)\big)\Big)\leqslant{\alpha}\quad\equi\quad 
\text{ I has a winning strategy in the game } G_\alpha\big[(z_n)\big].
}
\end{lemme}

\begin{proof}
A straightforward verification shows that 
\maths{
\rho\Big(\Psi \big((z_n)\big)\Big)>{\alpha}\quad\equi\quad 
\text{ II has a winning strategy in the game } G_\alpha\big[(z_n)\big].
}
Note also that, because the rules and winning conditions in both games are Borel, the games are determined, that is, either player I or II has a winning strategy \cite[Theorem 20.6]{Kechris}. Therefore,
\maths{
\rho\Big(\Psi \big((z_n)\big)\Big)\leqslant{\alpha}
&\quad\equi\quad 
\text{ II has no winning strategy in the game } G_\alpha\big[(z_n)\big]\\
&\quad\equi\quad 
\text{ I has a winning strategy in the game } G_\alpha\big[(z_n)\big]
}
as claimed.
\end{proof}

\begin{defi}\label{defi:alphaFatou}
The separable Banach lattice $X$ is said to be {\em $\alpha$-Fatou} provided that, for every $(z_n)\in X_\downarrow$, 
\maths{
\inf_n z_n=0
&\quad\equi\quad\text{ I has a winning strategy in the game } G_\alpha\big[(z_n)\big]\\
&\quad\equi\quad\rho\Big(\Psi \big((z_n)\big)\Big)\leqslant{\alpha}.\
}
\end{defi}

By Example \ref{example:fatou}, Banach lattices with the Fatou property are $1$-Fatou. 

\begin{proof}[Proof of Theorem \ref{thm:Fatou}]
(1)$\saa$(2): Suppose that (1) holds, i.e., that $X_{\downarrow0}$ is Borel. Then, for all sequences $(x_n)$ and vectors $x$, we have
\maths{
x_n\conv{\sigma\sf o}{n}x
&\quad\equi\quad \e (z_n)\in X_{\downarrow0}\; \;\a m\;\; \a^\infty n\;\; |x_n-x|\leqslant z_m,
}
which is clearly ${\bf \Sigma}^1_1$.

(2)$\saa$(1): Observe that, for a sequence $x_1\geqslant x_2\geqslant \ldots\geqslant 0$, we have 
$$
\a y>0\; \e n \;\;y\not\leqslant x_n\quad\equi\quad \inf_nx_n=0 \quad\equi\quad x_n\conv{\sigma\sf o}{n}0.
$$
The first expression is clearly ${\bf \Pi}^1_1$. Hence, if the last is ${\bf \Sigma}^1_1$, then these equivalent expressions are all Borel and hence  $X_{\downarrow0}$ is a Borel set.

(1)$\equi$(3):
Just observe that, by Lemma \ref{lemma:BONS}, the set $X_{\downarrow0}$ is Borel if and only if there is some $\alpha<\om_1$ so that
$$
\rho\Big(\Psi \big((z_n)\big)\Big)\leqslant{\alpha}
$$
for all $(z_n)\in X_{\downarrow0}$.
\end{proof}


\section{Examples of spaces with higher order Fatou properties}\label{sec:Fatou-ex}
Our next task is to show that the hierarchy of $\alpha$-Fatou properties does not collapse. That is, we will construct spaces that are $\alpha$-Fatou, but only for larger and larger $\alpha<\om_1$.

\begin{thm} 
For every countable ordinal $\alpha$, there is a separable Banach lattice with a countable $\pi$-basis that fails to be $\alpha$-Fatou.
\end{thm}

\begin{proof}
Our proof goes by induction on $1\leqslant \alpha<\om_1$. For each $\alpha$, we will construct several objects,
\begin{enumerate}
    \item a separable Banach lattice $(X_\alpha,\norm{\cdot}_\alpha)$,
    \item a countable $\pi$-basis $B_\alpha$ for $X_\alpha$,
    \item a Banach lattice homomorphism $X_\alpha\maps{\phi_\alpha} \mathbb{R}$ of norm $1$,
    \item a sequence 
$$
z^\alpha_1\geqslant z^\alpha_2\geqslant  \ldots> 0=\inf_nz^\alpha_n
$$
whose terms belong to the set 
$$
S_\alpha=\big\{x\in X^+_\alpha\del \norm{x}_\alpha=\phi_\alpha(x)=1\big\},
$$
\item\label{last property} a countable tree $T_\alpha\subseteq \Psi\big((z^\alpha_n)\big)\cap S_\alpha^{<\N}$ so that
$$
\rho(T_\alpha)>\alpha.
$$
\end{enumerate}
Observe then that, by property \eqref{last property}, we have
$$
\rho\Big(\Psi\big((z^\alpha_n)\big)\Big)>{\alpha}
$$
and hence that $X_\alpha$ fails to be $\alpha$-Fatou.

\

{\bf Base case, $\alpha=1$.} We let $X_1$ be the Banach lattice  $c$ of convergent real sequences endowed with the equivalent renorming 
$$
\Norm{(t_1,t_2,\ldots)}_1 =\max\Big\{ \tfrac{1}{3}\Norm{(t_1,t_2,\ldots)}_{\ell_\infty},  \big|\lim_nt_n\big|\Big\}.
$$ 
Note that 
$$
B_1=\Mgd{(0,\ldots, 0,t,0,\ldots)}{t\in \Q_+}
$$
is a countable $\pi$-basis for $X_\alpha$.
Define $\phi_1$ by
$$
\phi_1\big((t_1,t_2,\ldots)\big) = \lim_n t_n
$$
and, for each $n$, set 
$$
z^1_n = (\underbrace{0,0,\ldots,0}_{n\text{ times}}, 1,1,1,\ldots) \in S_1.
$$ 
Finally, letting $y_1 = (1,1,1,\ldots)\in S_1$, we see that $(y_1)\in \Psi\big((z^\alpha_n)\big)$ and therefore that
$$
T_1=\big\{\tom, (y_1)\big\}\subseteq \Psi\big((z^\alpha_n)\big)\cap S_1^{<\N}
$$
and that $\rho(T_1)=2>1$.

\

{\bf Successor case.} Suppose that $(X_\alpha,\norm{\cdot}_\alpha)$, $B_\alpha$, $\phi_\alpha$, $(z_n^\alpha)$ and $T_\alpha$ have been defined as above. We then let $X_{\alpha+1}=X_\alpha$ with the new norm 
$$
\norm x_{\alpha+1} = \max\Big\{\tfrac17{\norm x_\alpha},\big|\phi_\alpha(x)\big|\Big\}.
$$ 
Observe that $\norm\cdot_{\alpha+1}\leqslant \norm\cdot_\alpha$, whereby $\norm{\phi_\alpha}_{\alpha+1}=1$ and we may therefore set 
 $\phi_{\alpha+1} =\phi_\alpha$ and $B_{\alpha+1}=B_\alpha$.
 Note also that
$$
\norm x_{\alpha+1}=\norm x_\alpha \;\equi\;  \norm x_\alpha=|\phi_\alpha(x)|,
$$
from which it follows that $S_{\alpha}\subseteq S_{\alpha+1}$. We may thus set $z^{\alpha+1}_n = z^\alpha_n\in S_\alpha\subseteq S_{\alpha+1}$ for all $n$.
Finally,  let 
$$
T_{\alpha+1}
=\big\{ (z_1^\alpha,y_1,\ldots, y_n)\del (y_1,\ldots, y_n)\in T_\alpha\big\}\cup\{\tom\} \;\;\subseteq \;\; S_{\alpha}^{<\N} \;\;\subseteq \;\; S_{\alpha+1}^{<\N}
$$
and note that $\rho(T_{\alpha+1})>\rho(T_{\alpha})>\alpha$.
To see that $T_{\alpha+1}\subseteq  \Psi\big((z^{\alpha+1}_n)\big)$, note first that, for all $x,y\in S_{\alpha}$, we have 
$$
\phi_\alpha\big((x-y)^+\big)
=\big(\phi_\alpha(x-y)\big)^+
=0^+
=0
=\phi_\alpha(x-y),
$$
whereby $\Norm{(x-y)^+}_{\alpha+1}=\tfrac17\Norm{(x-y)^+}_{\alpha}$ and $\norm{x-y}_{\alpha+1}=\tfrac17\norm{x-y}_{\alpha}$.
Thus, for all $n$, 
\maths{
\Norm{(z_1^{\alpha} - z^{\alpha+1}_n)^+}_{\alpha+1} 
&=\Norm{(z_1^{\alpha} - z^{\alpha}_n)^+}_{\alpha+1} \\
&=\tfrac{1}{7}\Norm{(z_1^{\alpha} - z^{\alpha}_n)^+}_\alpha\\
&\leqslant\tfrac{1}{7}\norm{z^\alpha_1}_\alpha\\
& =\tfrac{1}{7}\norm{z^\alpha_1}_{\alpha+1}\\
& < \frac{\norm{z_1^{\alpha}}_{\alpha+1}}{3^1}
}
and 
\maths{
\Norm{(y_i - z^{\alpha+1}_n)^+}_{\alpha+1} 
&=\Norm{(y_i - z^{\alpha}_n)^+}_{\alpha+1} \\
&= \tfrac{1}{7}\Norm{(y_{i}-z^\alpha_n)^+}_\alpha\\
&\leqslant  \frac{1}{7}\frac{\|{y}_{i}\|_\alpha}{3^{i}} \\
&<  \frac{\norm{y_i}_{\alpha+1}}{3^{i+1}}.
} 
Similarly, 
\maths{
\norm{y_1 -z_1^{\alpha} }_{\alpha+1}
& = \tfrac17\Norm{y_1-z^\alpha_1}_\alpha \leqslant \frac{2}{7}<  \frac{\norm{z_1^\alpha}_{\alpha+1}}{3^1}
}
and 
\maths{
\norm{{y}_{i+1} - {y}_i}_{\alpha+1} 
= \tfrac17\Norm{y_{i+1}-y_{i}}_\alpha
\leqslant \frac17\frac{\norm{y_i}_\alpha}{3^{i}} 
< \frac{\norm{y_i}_{\alpha+1}}{3^{i+1}}.
} 
Together, these inequalities establish that $(z_1^\alpha,y_1,\ldots, y_n)\in \Psi\big((z^{\alpha+1}_n)\big)$ whenever $(y_1,\ldots, y_n)\in T_\alpha$ and hence that $T_{\alpha+1}\subseteq \Psi\big((z^{\alpha+1}_n)\big)$.

	\
	
{\bf Limit case.} Suppose that $\alpha_1 <\alpha_2 < \ldots<\alpha = \lim_n \alpha_n$. Assume also that the construction has been done for all ordinals smaller than $\alpha$. We first consider the $\ell_\infty$-sum of the Banach lattices $X_{\alpha_n}$,
$$
X^\infty_\alpha 
= \bigoplus_{\ell_\infty}\MGd{X_{\alpha_n} }{ n\in \N} 
= \MGd{x=(x^1,x^2,\ldots)\in \prod_{n\in \N}X_{\alpha_n} }
{ \sup_n\|x^n\|_{\alpha_n}<\infty}.
$$
Endowed with coordinatewise operations and the norm $\|x\|_\alpha = \sup_n \|x^n\|_{\alpha_n}$, this is a Banach lattice. Furthermore, because each $\phi_{\alpha_n}$ is a Banach lattice homomorphism of norm 1, it follows that 
$$
X^\phi_\alpha =\MGd{x=(x^1,x^2,\ldots)\in X^\infty_\alpha }{ \lim_n \phi_{\alpha_n}(x^n) \text{ exists}}
$$ 
is a Banach sublattice of $X^\infty_\alpha$ and that $\phi\colon X^\phi_\alpha\to \R$ defined by
$$
\phi\big((x^n)\big)=\lim_n\phi_{\alpha_n}(x^n)
$$
is a lattice homomorphism of norm $1$. Observe also that $$X^\infty_\alpha\cap \prod_n S_{\alpha_n}\subseteq X^\phi_\alpha.
$$
Unfortunately, there is no reason for $X^\phi_\alpha$ to be separable and so $X_\alpha$ will be chosen to be a specific  separable Banach sublattice of $X^\phi_\alpha$.

For this we first define the uncountable tree
$T^\phi_\alpha$ of all finite strings of elements of $X^\phi_\alpha$ of the form 
$$
\Big(
\big(\underbrace{0,0,\ldots,0}_{k-1},{y_1^k},{y_1^{k+1}},{y_1^{k+2}},\ldots\big)
\,,\;\; \ldots\;\; , \, 
\big(\underbrace{0,0,\ldots,0}_{k-1},{y_n^k},{y_n^{k+1}},{y_n^{k+2}},\ldots\big)
\Big),
$$
where $k\geqslant 1$ and $(y_1^m,y_2^m,\ldots, y_n^m)\in T_{\alpha_m}$ for all $m\geqslant k$. We observe that $\rho(T^\phi_\alpha)\geqslant \alpha+1.$ 
This is because the rank of an element of the tree as above is given by 
$$
\min\Mgd{\rho_{T_{\alpha_m}}\big((y_1^m,y_2^m,\ldots, y_n^m)\big) } { m\geqslant k},
$$ 
and therefore, by the inductive hypothesis, we can find strings of rank at least $\alpha_m$ for every $m$, and the rank of the root $\emptyset$ is $\alpha$. 

By induction on countable ordinals $\beta$, it is easily seen that, if a tree $\Upsilon$ satisfies $\rho(\Upsilon)\geqslant \beta$, then there is a countable subtree $\Upsilon'\subseteq \Upsilon$ with $\rho(\Upsilon')\geqslant \beta$. 
Applying this to $\Upsilon=T^\phi_\alpha$, we take a countable tree $T_\alpha\subseteq T^\phi_\alpha$ such that $\rho(T_\alpha)>\alpha$.

We define $X_\alpha$ to be the separable Banach sublattice of $X^\phi_\alpha$ generated by 
\begin{enumerate}
    \item all elements of $X^\phi_\alpha$ that appear in the tree $T_\alpha$, 
    \item all eventually zero vectors $(x_1,\ldots,x_n,0,0,\ldots)\in X^\phi_\alpha$,
    \item vectors $z^\alpha_k = \left({z^{\alpha_1}_k},{z^{\alpha_2}_k},{z^{\alpha_3}_k},\ldots\right)$ for  $k<\omega$.  
\end{enumerate}

If $B_{\alpha_n}$ is the countable $\pi$-basis for $X_{\alpha_n}$, then the collection $B_\alpha$ of vectors of the form $(0,\ldots,0,b,0,0,\ldots)$
with $b\in B_{\alpha_n}$ in the $n$th position form a countable $\pi$-basis for $X_\alpha$. The sequence $z^\alpha_k = \left({z^{\alpha_1}_k},{z^{\alpha_2}_k},{z^{\alpha_3}_k},\ldots\right)$ that we already defined is contained in $S_\alpha$ and satisfies 
$$
z^\alpha_1\geqslant z^\alpha_2\geqslant \ldots> 0=\inf_kz^\alpha_k.
$$

The nodes of the tree $T_\alpha$ indeed belong to $S_\alpha^{<\mathbb{N}}$ and by construction we have that $\rho(T_\alpha)>\alpha$. It thus remains to show that $T_{\alpha}\subseteq  \Psi\big((z^{\alpha}_m)\big)$. 
First,
\maths{
\BNORM{
\Big(\underbrace{0,0,\ldots,0}_{k-1}&,{y_{i+1}^k},y_{i+1}^{k+1},{y_{i+1}^{k+2}},\ldots\Big)
-
\Big(\underbrace{0,0,\ldots,0}_{k-1},{y_i^k},{y_i^{k+1}},{y_i^{k+2}},\ldots\Big)
}_\alpha
\\
&
=\BNORM{
\Big(\underbrace{0,0,\ldots,0}_{k-1},  \big(  {y_{i+1}^k-y_i^k}\big),   \big(  {y_{i+1}^{k+1}-y_i^{k+1}}\big),    \big(y_{i+1}^{k+2}-y_i^{k+2}\big),\ldots\Big)
}_\alpha
\\
&
=\sup_{m\geqslant k}\NORM{ {y_{i+1}^m-y_i^m} }_{\alpha_m}
\\
&
\leqslant \sup_{m\geqslant k}\tfrac 1{3^i}\Norm{ {y_i^m} }_{\alpha_m}
\\
&
\leqslant \frac 1{3^i}
\NORM{
\Big(\underbrace{0,0,\ldots,0}_{k-1},{y_{i}^k},{y_{i}^{k+1}},{y_{i}^{k+2}},\ldots\Big)
}_\alpha.
}
Similarly, for all $k$ and $m$,
\maths{
\BNORM{
\Bigg(&\Big(\underbrace{0,0,\ldots,0}_{k-1},{y_{i}^k},{y_{i}^{k+1}},{y_{i}^{k+2}},\ldots\Big)
-
z^\alpha_m\Bigg)^+
}_\alpha
\\
&
=\BNORM{
\Bigg(\Big(\underbrace{0,0,\ldots,0}_{k-1},{y_{i}^k},{y_{i}^{k+1}},{y_{i}^{k+2}},\ldots\Big)
-
\Big({z^{\alpha_1}_m},{z^{\alpha_2}_m},{z^{\alpha_3}_m},\ldots\Big)\Bigg)^+
}_\alpha
\\
&
=\BNORM{\Big(
\underbrace{0,0,\ldots,0}_{k-1},
{\big({y_{i}^k}-{z^{\alpha_k}_m}\big)^+},
{\big({y_{i}^{k+1}}-{z^{\alpha_{k+1}}_{m}}\big)^+},{\big({y_{i}^{k+2}}-{z^{\alpha_{k+2}}_{m}}\big)^+},\ldots\Big)
}_\alpha
\\
&
=\sup_{r\geqslant k}    \NORM{        \big(    y_{i}^r      -     z^{\alpha_r}_m     \big)^+            }_{\alpha_r} 
\\
&
\leqslant \sup_{r\geqslant k} \tfrac1{3^i}\Norm{y_{i}^r}_{\alpha_r} 
\\
&
\leqslant \frac 1{3^i}
\NORM{
\Big(\underbrace{0,0,\ldots,0}_{k-1},{y_{i}^k},{y_{i}^{k+1}},{y_{i}^{k+2}},\ldots\Big)
}_\alpha.
} 
This completes the verification of properties (1)-(5) and hence the inductive step.
\end{proof}

\end{document}